\documentclass[twoside,a4paper,11pt]{article}
\usepackage{amsfonts, amsbsy, amsmath, amsthm, amssymb, latexsym}
\usepackage{mathrsfs}
\usepackage{enumerate,centernot}
\usepackage[top=30mm,right=30mm,bottom=30mm,left=30mm]{geometry}

\usepackage{bm}
\usepackage{fancyhdr}

\numberwithin{equation}{section}

\newcommand{\F}{\mathbb{F}}
\newcommand{\Q}{\mathbb{Q}}
\newcommand{\C}{\mathbb{C}}
\newcommand{\N}{\mathbb{N}}
\newcommand{\cO}{\mathcal{O}}
\newcommand{\CG}{\mathscr{G}}
\newcommand{\CH}{\mathscr{H}}
\newcommand{\CX}{\mathscr{X}}
\newcommand{\CY}{\mathscr{Y}}
\newcommand{\CZ}{\mathscr{Z}}

\renewcommand{\bf}{\textbf}

 \renewcommand{\to}{\rightarrow}

\newcommand{\Hom}{{\rm Hom}}

\newcommand{\Spec}{\operatorname{Spec}}
\newcommand{\Tr}{\operatorname{Tr}}
\newcommand{\Ad}{\mathrm{Ad}}

\newcommand{\uG}{\underline{G}}
\newcommand{\uH}{\underline{H}}
\newcommand{\uK}{\underline{K}}

\newcommand{\uY}{\underline{Y}}
\newcommand{\uZ}{\underline{Z}}
\newcommand{\R}{\mathbb{R}}
\newcommand{\Z}{\mathbb{Z}}

\newtheorem{thm}{Theorem}[section]

\newtheorem{prop}[thm]{Proposition}
\newtheorem{lem}[thm]{Lemma}
\newtheorem{cor}[thm]{Corollary}

\newtheorem{rmk}[thm]{Remark}

\theoremstyle{definition}
\newtheorem{defn}[thm]{Definition}
\newtheorem{remk}[thm]{Remark}

\newtheorem{cla}[thm]{Claim}

\newtheorem*{ack}{Acknowledgments}

\begin{document}

\title{Deformation theory\\ and finite simple quotients of triangle groups I}

\author{Michael Larsen, Alexander Lubotzky, Claude Marion}

\date{\today}
\maketitle

\textbf{Abstract}
Let $2 \leq a \leq b \leq c \in \mathbb{N}$ with $\mu=1/a+1/b+1/c<1$ and let $T=T_{a,b,c}=\langle x,y,z: x^a=y^b=z^c=xyz=1\rangle$ be the corresponding hyperbolic triangle group. Many papers have been dedicated to the following question: what are the finite (simple) groups which appear as quotients of $T$? (Classically, for $(a,b,c)=(2,3,7)$ and more recently also for general $(a,b,c)$.) These papers have used either explicit constructive methods or probabilistic ones. The goal of this paper is to present a new approach based on the theory of representation varieties (via deformation theory). As a corollary we essentially prove a conjecture of Marion \cite{Marionconj}  showing that various finite simple groups are not quotients of $T$, as well as positive results showing that many finite simple groups are quotients of $T$.

\section{Introduction}

Let $a,b,c$ be a triple of positive integers. A group $G$ is said to be an $(a,b,c)$-group if it is generated by two elements of orders dividing $a$ and $b$ respectively, whose product has order dividing $c$, in other words, it is a quotient of the triangle group
\begin{equation}\label{e:tg}
T=T_{a,b,c}=\langle x, y, z:x^a=y^b=z^c=xyz=1\rangle.
\end{equation}

Many papers have been devoted to the question of understanding which (finite) groups are $(a,b,c)$-groups and especially which finite simple groups are. If $1/a+1/b+1/c\geq 1$ then $T$ is soluble or isomorphic to the alternating group ${\rm Alt}_5$ and the finite quotients of $T$ are well understood (see \cite{Conder1990}). We assume
\begin{equation*}
\mu=\frac1a+\frac1b+\frac1c<1,
\end{equation*}
so that $T$ is a cocompact Fuchsian group (of genus 0) and more specifically a hyperbolic triangle group. We call $(a,b,c)$ a \textit{hyperbolic triple of integers} and without loss of generality,  we suppose $a\leq b \leq c$. Recall that
\begin{equation}\label{e:hypub}
\mu=\frac1a+\frac1b+\frac1c\leq \frac{41}{42},
\end{equation}
where the upper bound $41/42$ is attained only if $(a,b,c)=(2,3,7)$.

A considerable effort has been made to try to classify the finite $(2,3,7)$-groups, also referred as \textit{Hurwitz groups}; for a recent survey, see \cite{Conder}. Recently, attention has also  been given to  other hyperbolic triples $(a,b,c)$, see for example \cite{Marionpsl2,Marionconj,LLM,Marionex,Marionpsl3one,Marionpsl3two,Marionpsl3three,FMP, GLL,GM}, where deterministic and probabilistic results on $(a,b,c)$-generation of finite simple groups of Lie type are obtained mainly in the special case where $a$, $b$ and $c$ are  prime numbers. Turning to general hyperbolic triples of integers, any finite simple group, being 2-generated, is a quotient of some triangle group $T$ and in fact can be so realized in many independent ways. See for example   \cite{GLL,FMP,GM} and the references therein establishing that every finite simple group other than ${\rm Alt}_5$ admits an (unmixed) Beauville structure.

The vast literature showing that some finite (simple) groups $G$ are $(a,b,c)$-groups has so far followed two main lines: either one gives two explicit generators of orders dividing $a$ and $b$ and whose product has order dividing $c$, or one uses probabilistic methods to show that such generators exist. In the latter approach, one typically uses character-theoretic methods to estimate the number of homomorphisms
from $T_{a,b,c}$ to $G$
and then
uses a knowledge of the maximal subgroups of $G$ to get a lower bound on the number of these homomorphisms which are surjective---see \cite{GLL} for a typical example.

In this paper we present a third method to prove (or disprove) that various groups are $(a,b,c)$-groups. Our method is based on deformation theory of representation varieties.
In a previous paper \cite{LL}, we used similar methods to study the representation variety $\Hom(\Gamma, G)$ where $\Gamma$ is a general Fuchsian group and $G$ is a quasisimple real Lie group.
In the current paper we use these methods to study the representation variety $\Hom(\Gamma,\underline{G})$ where this time $\Gamma=T=T_{a,b,c}$ is a (hyperbolic) triangle group and $\uG$ denotes a  quasisimple  algebraic group  defined over a field $\F$ (i.e., a semisimple algebraic group over $\F$ which is absolutely simple modulo its center).

A key point is: there exist such an infinite field $\mathbb{F}$ and a representation $\rho \in \Hom(T,\underline{G}(\F))$ with a Zariski dense image if and only if for infinitely many $q$, $\underline{G}(q)$ is an $(a,b,c)$-group. Moreover if ${\rm char}(\mathbb{F})=0$ and if in addition $\rho$ is not locally rigid then $T$ is \textit{saturated with such finite (simple) quotients} (see Definition \ref{d:sat} below).

Given an irreducible Dynkin diagram  $X$, we say that a quasisimple algebraic group $\uG$ defined over a field $\mathbb{F}$ is \emph{of type $X$}, if the associated diagram is $X$. By a slight abuse of notation, we write $X(\mathbb{F})$ for $\uG(\mathbb{F})$ where $\uG$ is the adjoint simple form, and for finite fields $\mathbb{F}=\mathbb{F}_q$,  we write $X(q)$ for the finite simple (untwisted) Chevalley group of type $X$ over $\mathbb{F}_q$. If $G$ is a  simple compact Lie  group, then $G=\uG(\mathbb{R})$ for some simple $\mathbb{R}$-group $\uG$ and in this case we say that $G$ is of type $X$ if $\uG$ is of type $X$.

\begin{defn}\label{d:sat}
Given an irreducible  Dynkin diagram  $X$, we say that $T$ is saturated with finite  quotients of type $X$ if there exist $p_0$ and $e$ in $\mathbb{N}$ such that for all  $p>p_0$, the finite simple group $X(p^{e\ell})$ is a quotient of $T$ for every $\ell \in \mathbb{N}$ and for a set of positive density of primes $p$, we even have $X(p^\ell)$ is a quotient of $T$ for every $\ell \in \mathbb{N}$.
\end{defn}

 Our main result is the following:

\begin{thm}\label{t:main}
For every hyperbolic triangle group $T=T_{a,b,c}$ and every irreducible Dynkin diagram $X$, $T$ is saturated with finite  quotients of type $X$, except possibly if
$(X,T)$ appears in Table \ref{tab:main} below.
\end{thm}

Let us say right away that Theorem \ref{t:main} is not the best result we can get via this method. We chose to illustrate in this paper the main point, which is the relevance of deformation theory to the problem of characterizing finite simple quotients of triangle groups. In a second paper, we will exploit the method further to get stronger results, paying the price of  being more technical. In particular in \cite{LLM2} we show that the six possibly exceptional  hyperbolic triangle groups in Table \ref{tab:main}, namely those in
$$S=\{T_{2,4,6},T_{2,6,6},T_{2,6,10},T_{3,4,4},T_{3,6,6}, T_{4,6,12}\}$$
are not really exceptions.

To prove the theorem we want to get Zariski dense representations of $T$ into a group $\uG$ of type $X$ which are not locally rigid. To this end we will start with a representation of $T$ to  ${\rm SO}(3,\R)$ and then  we will compose with the principal homomorphism from ${\rm SO}(3,\R)$ to  a compact simple Lie group $G$ of type $X$ and deform the resulting (non-dense) homomorphism
$T\to G$. For this reason we have to exclude the six triangle groups in $S$, which are the (only) hyperbolic triangle groups without ${\rm SO}(3,\R)$-dense representations (see \cite{LL}). In \cite{LLM2} we will push forward the general method in order to include these six groups and to eliminate some more cases of Table \ref{tab:main}. The way to obtain a Zariski dense representation of $T$ to  $G=\uG(\mathbb{R})$ is by deforming the representation $T\rightarrow G$ induced from the principal homomorphism ${\rm SO}(3,\R) \rightarrow G$. This is done in one step if ${\rm SO}(3,\R)$ is maximal in $G$, but in some cases we have to do it in two or even three steps (through ``steps in the ladder"---see \S\ref{s:dph}).

Going through Table \ref{tab:main} we can deduce:

\begin{table}
\caption{Possible exceptions  to Theorem \ref{t:main}}\label{tab:main}
\center
\begin{tabular}{|l|l|l|}
\hline
$X$& $(a,b,c)$ & $r$\\
\hline
any & $(2,4,6)$, $(2,6,6)$, $(2,6,10)$ & \\
& $(3,4,4)$, $(3,6,6)$, $(4,6,12)$ & \\
\hline
$A_r$  & $(2,3,7)$ & $r \leq 19$ \\
& $(2,3,8)$ & $r \leq 13$\\
& $(2,3,c)$, $c \geq 9$ & $r\leq 7$\\
& $(2,4,5)$ & $r\leq 13$\\
& $(2,4,c)$, $c \geq 7$& $r \leq 5$\\
& $(2,5,5)$ & $r=6$\\
& $(2,b,c)$, $b \geq 5$ & $r\leq3$\\
& $(3,3,c)$, $c \geq 4$ & $r \in \{3,4,6\}$\\
& any & $r=1$\\
\hline
$B_3$ & $(2,3,c)$, $c\geq 7$ & \\
& $(3,3,c)$, $c \geq 4$ &\\
& $(2,4,5)$ & \\
& $(2,5,5)$ & \\
\hline
$C_2$ &  $(2,3,c)$, $c\geq 7$ & \\
& $(3,3,c)$, $c \geq 4$ & \\
\hline
 $D_r$ & $(2,3,7)$ & $r \in\{4,5,7,8,9,10,11,13,15,16,17,19,22,23,25,29,31,37,43\}$\\
$(r\geq 4)$& $(2,3,8)$& $r\in \{4,5,7,9,10,11,13,17,19,25\}$ \\
& $(2,3,9)$ & $r\in \{4,5,7,10,11,13,19\}$\\
& $(2,3,10)$ & $r\in \{4,5,7,11,13\}$\\
& $(2,3,11)$& $r \in\{4,5,7,13\}$\\
& $(2,3,12)$& $r \in\{4,5,7,13\}$\\
& $(2,3,c)$, $c\geq 13$ & $r\in \{4,5,7\}$\\
& $(2,4,5)$ & $r \in \{4,5,6,7,9,11,13,17,21\}$\\
& $(2,4,7)$ & $r \in \{5,9\}$\\
& $(2,4,8)$ & $r\in \{5,9\}$\\
& $(2,4,c)$, $c \geq 9$ & $r=5$\\
& $(2,5,5)$ & $r \in \{4,6,7,11\}$\\
& $(2,5,6)$ & $r=7$\\
& $(3,3,4)$ & $r \in\{4,5,7,10,13\}$\\
& $(3,3,5)$ & $r \in\{4,7\}$\\
& $(3,3,6)$ & $r =\{4,7\}$\\
& $(3,3,c)$, $c\geq 7$ & $r=4$\\
& $(4,4,4)$& $r=5$\\
\hline
$G_2$ &$(2,4,5)$ & \\
& $(2,5,5)$& \\
\hline
$E_6$ & $(2,3,c)$, $c \in \{7,8\}$ & \\
& $(2,4,c)$, $c \in \{5,7,8\}$& \\
\hline
\end{tabular}
\end{table}

\begin{cor}\label{cor:first}
If $\mu=1/a+1/b+1/c<1/2$ then for every irreducible Dynkin diagram $X\neq A_1$, $T_{a,b,c}$ is saturated with finite  quotients of type $X$.
\end{cor}

\begin{cor}
Assume $X \not \in \{A_r:1\leq r\leq 7\}\cup \{B_3\}\cup\{C_2\}\cup\{D_r: r=4,5,7\}$. Then for almost every hyperbolic triple $(a,b,c)$, the group $T=T_{a,b,c}$ is saturated with finite  quotients of type $X$.
\end{cor}

Many of our results are new even in the classical case $(a,b,c)=(2,3,7)$.

\begin{cor}
$T_{2,3,7}$ is saturated with finite  quotients of type $X$ except possibly if $X \in S$ where
\begin{multline*}
S =   \{A_r:1\leq r\leq 19\}\cup \{B_3\}\cup\{C_2\}\\
\cup\{D_r: r\in \{4,5,7,8,9,10,11,13,15,16,17,19,22,23,25,29,31,37,43\}\}\cup\{E_6\}.
\end{multline*}
 In particular, it is saturated with finite quotients of type $E_8$.
\end{cor}

This last sentence answers a question we were asked by Guralnick.

Let us however mention one weakness of our method: it uses a non-explicit deformation $\rho$ of  a starting representation $\rho_0$. We therefore do not have a good control on the ring of definition of $\rho$. As a result,  we cannot give an explicit upper bound for $p_0$ or $e$ in Definition \ref{d:sat}, and so our method cannot give a result of the kind proved in \cite{LTW} stating that for $r \geq 286$ every finite simple group of type $A_r$ is a quotient of $T_{2,3,7}$.

Finally, let us call the attention of the reader to a slightly surprising corollary of our work. In \cite{DTZ} (see also \cite{Marionconj}) it was shown that many simply connected finite groups of classical type are not Hurwitz. For example, if  $n \in \{4,6,8,10,12,14,16,18,22\}$, then  ${\rm Sp}_n(q)$ with $q$ odd is never Hurwitz. Our results show that if the corresponding simple versions ${\rm PSp}_n(q)$ of the above groups are considered instead, infinitely many of them are Hurwitz groups (for any fixed $n\neq 4$).

It is interesting to compare our Theorem \ref{t:main} with a conjecture of Liebeck and Shalev proposed in \cite{LS}  (and proved for Fuchsian groups of genus at least 2) which states that for any  Fuchsian group $\Gamma$ there exists an integer $r(\Gamma)$ such that if $G$ is a finite simple classical group of rank at least $r(\Gamma)$ then a randomly chosen homomorphism in $\Hom(\Gamma,G)$ is an epimorphism with probability tending to 1 as $|G|\to \infty$.

This is for the positive side. On the negative side we essentially prove a conjecture of Marion proposed in \cite{Marionconj}: in that paper he studied $(a,b,c)$-generation of finite quasisimple groups of Lie type in the special case where $(a,b,c)$ is a triple of primes. Denoting by $\delta_m^{\uG}$  the dimension of the subvariety $\uG_{[m]}$ of $\uG$ (defined over $\mathbb{F}$ of characteristic $p\geq 0$) consisting of elements of order dividing $m$, he showed that if $p>0$, and
\begin{equation*}
\delta_a^{\uG}+\delta_b^{\uG}+\delta_c^{\uG}<2 \dim \uG,
\end{equation*}
then a finite quasisimple group of the form $G=\underline{G}(p^\ell)$ is never an $(a,b,c)$-group. He also conjectured that if
\begin{equation}
\label{e:rigid}
\delta_a^{\uG}+\delta_b^{\uG}+\delta_c^{\uG}=2 \dim \uG,
\end{equation}
then there are only finitely many positive integers $\ell$ such that $\underline{G}(\F_{p^\ell})$ is an $(a,b,c)$-group.
We use the following definition, following \cite{Marionconj}:

\begin{defn}
A triple $(a,b,c)$ is \emph{rigid} for $\uG$ if (\ref{e:rigid}) holds.
\end{defn}

If (\ref{e:rigid}) does not hold,  $(a,b,c)$ is said to be \emph{reducible} (respectively, \emph{nonrigid}) for $\uG$ according as $\delta_a^{\uG}+\delta_b^{\uG}+\delta_c^{\uG}$ is less (respectively, greater) than $2\dim \uG$.

Note that this is not the same as Thompson's rigidity condition (see \cite{Vol}) but  is related to it (see \cite{SV,Marionconj,LLM}).

\begin{thm}\label{t:marconj}
Let $\uG/\F_p$ be a quasisimple algebraic group of type $X$ and let $d$ be the determinant of the Cartan matrix of $X$. If $(a,b,c)$ is rigid for $\uG$  and $p\nmid abcd$ then there are only finitely many positive integers $\ell$ such that $\uG(\mathbb{F}_{p^\ell})$ is a quotient of $T=T_{a,b,c}$.
Moreover, if this holds in the case that $\uG$ is adjoint, then only finitely many finite simple groups of type $X$ and characteristic $p$ are $(a,b,c)$-groups.
\end{thm}

So Marion's conjecture is  true except possibly if $p$ divides $abcd$.  For a given $T=T_{a,b,c}$ and a given $X$, this excludes only finitely many primes.

Theorem \ref{t:marconj} is proved by showing that under its hypotheses, the epimorphisms from $T$ to $\uG(\mathbb{F}_{p^\ell})$ are all locally rigid (when considered as elements in $\Hom(T,\uG(\overline{\mathbb{F}}_p))$). Hence there are only finitely many.


We should mention that Marion classified the rigid pairs $((a,b,c),\underline{G})$ and proved this conjecture for many of them by a case by case study together with the notion of linear rigidity defined in \cite{SV}.  Our approach gives a conceptual explanation and dispenses with the assumption that $a$, $b$ and $c$ are prime numbers.

\begin{ack}
The authors are grateful to the ERC, ISF and NSF for their support  and thank Bob Guralnick and Aner Shalev for some useful and helpful discussions.
\end{ack}

 \section{Deformation theory of Fuchsian groups}\label{s:dtfg}

 In his well-known paper Weil \cite{Weil} introduced the language of cohomology into deformation theory.  He proved:

 \begin{thm}\label{t:locrig}
 Let $\Gamma$ be a finitely generated group and $\uG$ be an algebraic group defined over a field $\mathbb{F}$  and with Lie algebra $\mathfrak{g}$. Let $\rho \in \Hom(\Gamma,\uG)$ be a representation and ${\rm Ad}\circ \rho$ be the representation induced on $\mathfrak{g}$. Suppose $H^1(\Gamma, {\rm Ad}\circ \rho)=0$. Then
 $\rho$ is locally rigid, i.e. there is a neighborhood of $\rho$ consisting entirely of conjugates of $\rho$ by elements of $\underline{G}$.
 More precisely, there exists a universal domain $K\supset F$ such that every $K$-point in some Zariski neighborhood of $\rho$ in $\Hom(\Gamma,\uG)$ lies in the $\uG(K)$-orbit of $\rho$.
 \end{thm}

Weil then showed how to compute the space  of cocycles for a general finite dimensional representation of a general Fuchsian group.  For conciseness, we will present here the computation only for hyperbolic triangle groups $T=T_{a,b,c}$ (see \cite{Weil} and \cite{LL} for the general case).

Let $\Gamma$ act via a representation $s$ on a vector space $V$. A 1-cocycle is a map $\phi: \Gamma \rightarrow V$ with
$$ \phi(\gamma\gamma')=\phi(\gamma)+s(\gamma)\phi(\gamma')$$
for all $\gamma, \gamma' \in \Gamma$.  It is a 1-coboundary if there exists $w \in V$ such that $$ \phi(\gamma)=w-s(\gamma)w$$ for every $\gamma \in \Gamma$. Let $Z^1(\Gamma,V)$ (respectively, $B^1(\Gamma,V)$) be the space of 1-cocycles (respectively, 1-coboundaries) and let $H^1(\Gamma,V)=Z^1(\Gamma,V)/B^1(\Gamma,V)$.  For $\Gamma=T=T_{a,b,c}$ we say that $\phi$ is \textit{parabolic}
if for every finite subgroup $S$ of $T$, $\phi|_S$ is a 1-coboundary. It is well-known that the maximal finite subgroups of $T$  are the conjugates of $\langle x \rangle$, $\langle y \rangle$ and $\langle z \rangle$ of (\ref{e:tg}).
It follows that if $p={\rm char}(\mathbb{F})$ does not divide $abc$ then every 1-cocycle is parabolic. Let
$\tilde{P}^1(T,V)$ be the space of parabolic 1-cocycles and $P^1(T,V)=\tilde{P}^1(T,V)/B^1(T,V)$.  Hence if $p$ does not divide $abc$ then $\tilde{P}^1(T,V)=Z^1(T,V)$ and $P^1(T,V)=H^1(T,V)$.

In \cite[\S 6, pp. 155--156]{Weil}, Weil computed  $\dim \tilde{P}^1(T,V)$ and  $\dim P^1(T,V)$ for a general representation $s$ of $T=T_{a,b,c}$ (in fact for every Fuchsian group but we specialize his equations to our case):

\begin{equation}\label{e:dimptildeone}
\dim \tilde{P}^1(T,V)=-d+i^*+e_x+e_y+e_z
\end{equation}
and

\begin{equation}\label{e:dimpone}
\dim P^1(T,V)=-2d+i+i^*+e_x+e_y+e_z
\end{equation}
where
$d=\dim V$, $i$ is the dimension of the space of invariants of $s$, $i^*$ is the dimension of the space of invariants of $s^*$ (the dual of $s$), and for $t\in\{x,y,z\}$,
$e_t={\rm rank}(I_d-s(t))$.

The following consequence will be used several times:

\begin{lem}\label{l:h1equal}
Let  $s_1,s_2: T=T_{a,b,c}\rightarrow {\rm GL}(V)$ be two representations lying in a common irreducible component of $\Hom(T,{\rm GL}(V))$ where $V$ is a vector space over a field of characteristic zero. If the dimensions of the spaces of invariants of $s_1$, $s_1^*$, $s_2$, $s_2^*$ are all  zero, then
$$ \dim H^1(T, s_1)=\dim H^1(T, s_2).$$
\end{lem}

\begin{proof}
For $j \in \{1,2\}$, let $V^{s_j(x)}$ (respectively, $V^{s_j(y)}$ and $V^{s_j(z)}$) be the fixed point space of $s_j(x)$ (respectively, $s_j(y)$ and $s_j(z)$) in  $V$.  Since $V$ is defined over a field of characteristic zero and the dimensions of the spaces of invariants of $s_j$ and $s_j^*$ are zero, (\ref{e:dimpone}) yields
$$ \dim H^1(T,s_j)=\dim V- (\dim V^{s_j(x)}+\dim V^{s_j(y)}+\dim V^{ s_j(z)}).$$
Since the restrictions of two representations in a common irreducible component to a cyclic subgroup are conjugate, we get
 $\dim V^{s_1(x)}=\dim V^{s_2(x)}$ (and similarly for $y$ and $z$); this yields the result.
\end{proof}

Let $\uH$ be a real form of ${\rm PGL}_2$ and let  $\rho: T =T_{a,b,c}\rightarrow \uH(\mathbb{R})$ be an $\uH$-dense representation (in the sense of \cite{LL}) i.e.
\begin{enumerate}[(i)]
\item $\rho(T)$ is Zariski dense in $\uH,$ and
\item  $\rho(x)$ (respectively, $\rho(y)$, $\rho(z)$) has order exactly $a$ (respectively, $b$, $c$).
\end{enumerate}
We note that every $T$ has such a representation to the split form ${\rm PGL}_2(\mathbb{R})$. In \cite{LL} it was shown that every $T$ not in \begin{equation}\label{e:trianglenotso3dense}
S=\{T_{2,4,6}, T_{2,6,6},T_{2,6,10},T_{3,4,4},T_{3,6,6},T_{4,6,12}\}\end{equation}
has  such a representation to ${\rm SO}(3,\R)$.

Let   $\rho: T=T_{a,b,c}\rightarrow \uH$ be such a representation and let $s={\rm Ad} \circ \rho$ be  the action on the Lie algebra of $\uH(\mathbb{C})$. Then the eigenvalues of $s(x)$ (respectively, $s(y)$, $s(z)$) are $1$, $\omega$, $\omega^{-1}$ where $\omega$ is a primitive $a$-th root (respectively, $b$-th, $c$-th root) of unity. Thus $e_x=e_y=e_z=2$. Recall that in general the adjoint representation of a simple group, in characteristic zero, is self-dual. Also as $T$ is Zariski dense in $\uH$, there are no invariants. Hence $i=i^*=0$. Since here $d=3$, (\ref{e:dimpone})
gives $\dim P^1(T, {\rm Ad}\circ \rho)=0$ and hence  $\dim H^1(T, {\rm Ad}\circ \rho)=0$. In particular, by Theorem \ref{t:locrig}, $T$ is locally rigid in ${\rm PGL}_2(\mathbb{C})$ and in ${\rm SO}(3)$ (when $\rho$ is an ${\rm SO}(3)$-dense representation).

Let us now move to  more general representations of $T=T_{a,b,c}$ into an algebraic group $\uG$. It is a well-known theorem of de Siebenthal \cite{dS} and Dynkin \cite{D1} that for every (adjoint) simple algebraic group $\uG$ defined over $\mathbb{C}$ there exists a conjugacy class of {\textit{principal}} homomorphisms ${\rm SL}_2 \rightarrow \uG$ such that the image of any nontrivial unipotent element of ${\rm SL}_2(\mathbb{C})$ is a regular unipotent element of $\uG(\mathbb{C})$. The restriction of the adjoint representation of $\uG$ to ${\rm SL}_2$ via the principal homomorphism is  a direct sum of $V_{2e_j}$ ($1\leq j\leq r$), where $r$ is the Lie rank of $\uG$, $e_1,\dots,e_r$ is the sequence of exponents of $\uG$ and $V_k$ denotes the $k$-th symmetric power of the two-dimensional irreducible representation of ${\rm SL}_2$ (see \cite{Ko}). This is a $k+1$-dimensional representation and hence
\begin{equation}\label{e:dimge}
\dim \uG=\sum_{j=1}^r(1+2e_j).
\end{equation}
Note that the homomorphism ${\rm SL}_2 \rightarrow {\rm Ad}(\uG)$ factors through ${\rm PGL}_2$.
The principal homomorphism induces, by restriction, a homomorphism from ${\rm SU}(2)$ to $\uG(\mathbb{C})^c$, where $\uG(\mathbb{C})^c$ denotes a maximal compact subgroup of $\uG(\mathbb{C})$. The group $\uG(\C)^c$ is actually isomorphic to $\uG^c(\mathbb{R})$ where $\uG^c$ is a compact real form of $\uG$. When $\uG$ acts on its Lie algebra, the latter homomorphism factors through ${\rm SO}(3)$, and we also call the resulting homomorphism ${\rm SO}(3) \rightarrow \uG(\mathbb{C})^c$, the principal homomorphism.

Let now $\rho_0: T \rightarrow {\rm PGL}_2 \rightarrow \uG$ be the representation induced from the principal homomorphism ${\rm PGL}_2\rightarrow \uG$ and consider the representation ${\rm Ad} \circ \rho_0$ of $T$ on the Lie algebra $\mathfrak{g}$ of $\uG$. The eigenvalues of ${\rm Ad}\circ \rho_0(x)$ are:

$$\omega^{-2e_j}, \omega^{2-2e_j},\dots,\omega^0,\dots,\omega^{2e_j-2}, \omega^{2e_j}$$ for $ j=1,\dots,r$ and where $\omega$ is a primitive root of unity of degree $2a$ (and similarly for ${\rm Ad}\circ\rho_0(y)$ and ${\rm Ad}\circ\rho_0(z)$ with $2b$ and $2c$, respectively).

One checks that
\begin{equation}\label{e:valex}
e_x=\dim \uG-\sum_{j=1}^r\left(1+2\left\lfloor \frac{e_j}a\right \rfloor\right),
\end{equation}
 and similarly for $e_y$ and $e_z$. As ${\rm PGL}_2$ has no invariants on $\mathfrak{g}$ (all the $V_{2e_i}$ are nontrivial), we deduce  from (\ref{e:dimpone}) and (\ref{e:valex}) the following result.

In the statement below $\uG$ denotes an absolutely simple algebraic group  of adjoint type defined over $\mathbb{R}$ of type $X$ and rank $r$, and  $\uH$  an absolutely simple form of ${\rm PGL}_2$ defined over $\mathbb{R}$. The result will be used in this paper for $\uH={\rm SO}(3)$ and $\uG$ a compact real form, whereas in \cite{LLM2} we will need it for the split case.

\begin{prop}\label{p:dimpone}
Let $T=T_{a.b.c}$ be a hyperbolic triangle group with an $\uH$-dense representation $T\rightarrow \uH$, where $\uH$ is a real form of ${\rm PGL}_2$. Let  $${\rm Ad}\circ\rho_0: T\rightarrow \uH \rightarrow \uG \rightarrow {\rm Aut}(\mathfrak{g})$$ where $\rho_0: T \rightarrow \uG$ is the representation induced from the principal homomorphism $\uH \rightarrow \uG$, and write $n_1=a$, $n_2=b$, $n_3=c$. Then
$$\dim H^1(T,{\rm Ad}\circ\rho_0)= \dim P^1(T,{\rm Ad}\circ\rho_0)=\dim \uG-\sum_{k=1}^3\sum_{j=1}^r\left(1+2\left\lfloor \frac{e_j}{n_k} \right \rfloor\right).$$
\end{prop}

In the next result,  $\uG$ again  denotes an  absolutely simple algebraic group of adjoint type defined over $\mathbb{R}$ of type $X$ and rank $r$.

\begin{lem}\label{l:ineq}
Let  $e_1,\dots, e_r$ be the exponents of $\uG$ and write $n_1=a$, $n_2=b$ and $n_3=c$. Then
\begin{equation}\label{e:ineq} \sum_{k=1}^3\sum_{j=1}^r \left(1+2\left\lfloor \frac{e_j}{n_k} \right\rfloor\right)<\dim \uG \end{equation}
except in the following cases:
\begin{enumerate}[(a)]
\item $X=A_1$.
\item $X=A_2$ and $n_1=2$.
\item $X= A_3$ and $n_1=2$, $n_2=3$.
\item $X=A_4$ and $n_1=2$, $n_2=3$.
\item $X=C_2$ and $n_2=3$.
\item $X=G_2$ and $n_1=2$, $n_3=5$.
\end{enumerate}
In particular, with the notation of  Proposition \ref{p:dimpone}, we have $\dim H^1(T, {\rm Ad}\circ \rho_0 )> 0$ unless $(X,(n_1,n_2,n_3))$ is one of the exceptional cases (a)--(f),  in which case $\dim H^1(T,{\rm Ad}\circ \rho_0)=0$.
\end{lem}

\begin{remk}
In all  the exceptional cases we get equality in (\ref{e:ineq})  as would be expected from Proposition \ref{p:dimpone}.
\end{remk}

\begin{proof}
Without loss of generality we may assume (when dealing with the non-exceptional cases) that $n_3 \leq 7$, since the left hand side of (\ref{e:ineq}) is monotonically decreasing in each $n_k$. We note that
$$ (1+2\lfloor {e_j}/{n_k}\rfloor) -\frac{2e_j+1}{n_k} \in [n_k^{-1}-1,1-n_k^{-1}]$$ depends only on $e_j$ modulo $n_k$. On the other hand by (\ref{e:dimge}),
$$ \sum_{j=1}^r \frac{2e_j+1}{n_k}= \frac{\dim \uG}{n_k},$$ so
\begin{eqnarray*}
\sum_{k=1}^3\sum_{j=1}^r (1+2\lfloor e_j/n_k\rfloor) & =  & \left(\frac1{n_1}+\frac1{n_2}+\frac1{n_3}\right)\dim \uG +\sum_{k,j} \left((1+2\lfloor e_j/n_k \rfloor)-\frac{2e_j+1}{n_k} \right)
\end{eqnarray*}
\begin{eqnarray}\label{eqn:ineq}
& \leq & \frac{41}{42}\dim \uG +\sum_{k=1}^3\sum_{j=1}^r \left( (1+2\lfloor e_j/n_k \rfloor)-\frac{2e_j+1}{n_k}\right),
\end{eqnarray}
where in the above inequality we used (\ref{e:hypub}).
The exponents of the different root systems are as follows (see, for example, \cite{Bourbaki}):
\begin{equation*}A_r: 1,2,\dots, r; \ B_r,C_r: 1,3,\dots,2r-1;\ D_r: 1,3,\dots,2r-3, r-1; \ E_6:1,4,5,7,8,11;
\end{equation*}
\begin{equation}\label{e:exp}
E_7: 1,5,7,9,11,13, 17;\ E_8:1,7,11,13,17,19,23,29;\ F_4:1,5,7,11 ;\ G_2: 1,5.
\end{equation}
We tabulate the value (respectively, maximum value) of
$$ \sum_{j=1}^r\left((1+2\lfloor e_j/n \rfloor)-\frac{2e_j+1}n\right)$$
for each root system (respectively, family of root systems) of exceptional (respectively, classical) type and for each $n\leq 7$.

\begin{center}
\begin{tabular}{|l|l|l|l|l|l|l|l|l|}
\hline
$n$ & $A$ & $B/C$ & $D$ & $E_6$ & $E_7$ & $E_8$ & $F_4$ & $G_2$\\
\hline
2 & 0 & -1& -3/2& -1 & -7/2& -4 & -2 & -1\\
3 & 0 & 2/3 & 2/3& -2& -4/3& -8/3& -4/3& -2/3\\
4& 1/4& -1/4& -1/4& 1/2& -1/4& -2 & -1& 1/2\\
5& 2/5 & 4/5 & 4/5& 2/5& 2/5& -8/5& 8/5& 6/5\\
6 & 2/3 & 1/3 & 0& -1& -7/6 & -4/3 & -2/3 & -1/3\\
7 & 6/7 & 8/7 & 8/7 & 6/7 & 0 & 4/7 & 4/7 & 0\\
\hline
\end{tabular}
\end{center}

This table together with (\ref{eqn:ineq}) immediately implies the lemma for $A_r$ when $r \geq 10$, $B_r$, $C_r$,  $D_r$ when $r\geq 9$, $E_7$ and $E_8$. This reduces us to a finite list of cases which can be checked by hand, yielding the exceptions (a)--(f) listed above.
\end{proof}


Finally, let us recall the final sentence of Weil in his paper  \cite{Weil} which gives (when specialized to our case of interest):

\begin{thm}\label{t:weil}
Let $\rho:  T=T_{a,b,c} \rightarrow \uG$ where $\uG$ is an algebraic group defined over an algebraically closed field of characteristic $p \geq 0$. Assume
\begin{enumerate}[(a)]
\item $p$ does not divide $abc$.
\item $({\rm Ad}\circ \rho)^*$, the coadjoint representation of $T$,  has no (nontrivial) invariants.
\end{enumerate}
Then $\rho$ has a nonsingular neighborhood in ${\rm Hom}(T,\uG)$ of dimension $-d+e_x+e_y+e_z$.
\end{thm}
We are now ready to put all the above information together and prove our first main result.

\section{Proof of Marion's conjecture}

Let us consider a general representation $\rho: T=T_{a,b,c} \rightarrow \uG$, where $\uG$ is a simple algebraic group defined over an algebraically closed field $\mathbb{F}$ of characteristic $p \geq 0$, and  the action ${\rm Ad}\circ \rho$ on the Lie algebra $\mathfrak{g}$ of $\uG$, where ${\rm Ad}$ denotes the adjoint representation of $\uG$.  In the notation given in (\ref{e:dimpone}), $d= \dim \mathfrak{g}$, $i$ (respectively, $i^*$) is the dimension of the space of invariants of ${\rm Ad}\circ \rho$ (respectively, $({\rm Ad}\circ \rho)^*$) on $\mathfrak{g}$ (respectively, $\mathfrak{g}^*$), and for $t \in \{x,y,z\}$, $e_t={\rm rank}(I_d-{\rm Ad}\circ \rho(t))$. In particular, we have
\begin{equation*}\label{e:lietogroup}
d=\dim \uG \quad \textrm{and} \quad e_t \leq \dim t^{\uG}
\end{equation*}
 with equality in the latter inequality if $p$ does not divide the order of $t$, where, for $t \in \{x,y,z\}$, $t^{\uG}$ denotes the conjugacy class of $\rho(t)$ in $\uG$.
Setting $\delta_a^{\uG(\F)}$ to be the dimension of the subvariety $\uG_{[a]}$ of $\uG$ consisting of elements of order dividing $a$ (and similarly for $\delta_b^{\uG(\F)}$, $\delta_c^{\uG(\F)}$), we therefore  obtain
\begin{equation*}
d = \dim \uG, \quad e_x \leq \delta_a^{\uG(\F)}, \quad e_y\leq \delta_b^{\uG(\F)} \quad \textrm{and} \quad e_z \leq \delta_c^{\uG(\F)}.
\end{equation*}
Thus in this case (\ref{e:dimpone}) gives:
\begin{equation*}
0 \leq \dim P^1(T,V) \leq-2\dim \uG+i+i^*+\delta_a^{\uG(\F)}+\delta_b^{\uG(\F)}+\delta_c^{\uG(\F)}.
\end{equation*}
Hence, if the pair $((a,b,c),\uG)$ is rigid (i.e. $\delta_a^{\uG(\F)}+\delta_b^{\uG(\F)}+\delta_c^{\uG(\F)}=2\dim {\uG}$) we get
\begin{equation*}
0 \leq \dim P^1(T, {\rm Ad \circ \rho}) \leq i+i^*.
\end{equation*}
In particular if $p \nmid abc$ and $i=i^*=0$ then $\rho$ is locally rigid. In summary we get:

\begin{prop}\label{p:locrig}
Suppose that $(a,b,c)$ is rigid for $\uG$. If $p$ does not divide $abc$ and $\rho: T=T_{a,b,c}\rightarrow \underline{G}$ is such that ${\rm Ad}\circ \rho$ and $({\rm Ad}\circ \rho)^*$ have no invariants, then $H^1(T, {\rm Ad}\circ \rho)=0$, and so $\rho$ is locally rigid.
\end{prop}

We are now  ready to  prove our version of Marion's conjecture (see Theorem \ref{t:marconj}).

\textit{Proof of Theorem \ref{t:marconj}.}
As $p\nmid d$, there is no trivial factor in the Jordan-H\"older series of $\Ad\circ \rho$ as a $\uG$-representation
(see \cite{Hiss}),
so, by \cite[\S13]{Steinberg}, the same holds at the level of
$\uG(\F_{p^\ell})$-representations when $\ell$ is sufficiently large. (In fact $\ell=1$ suffices, but the argument is more subtle and uses the fact that the adjoint representation is $p$-restricted under the hypothesis.) Thus, if $\rho:T\rightarrow \uG(\F_{p^\ell})$ is an epimorphism, $\Ad\circ \rho$ satisfies the hypothesis of Proposition \ref{p:locrig} and $\rho$ is locally rigid as an element of $\Hom(T,\uG)$. In other words a small open neighborhood of $\rho$ lies inside the orbit of $\rho$ under conjugation. This implies that the orbit of $\rho$ is open. A variety  can have at most finitely many such disjoint open sets, so this proves the first part of the theorem.

For the second one, if $G$ is a sufficiently large finite simple group of type $X$ in characteristic $p$, it can be regarded either as the derived group of $\uG(\F_{p^\ell})$ for a simple adjoint group $\uG$ or as the quotient of $\uG^{\mathrm{sc}}(\F_{q^\ell})$ by its center.  Every irreducible representation of $\uG$ can be regarded as an irreducible representation of $\uG^{\mathrm{sc}}$ and therefore, when $\ell$ is sufficiently large, as an irreducible representation of $\uG^{\mathrm{sc}}(\F_{q^\ell})$ which factors through $G$.  Thus if $\ell$ is sufficiently large, $G$ cannot be an $(a,b,c)$-group.
 \quad $\square$\\

Marion \cite{Marionconj} discussed the notions of reducible, rigid and nonrigid hyperbolic triples only over fields of positive characteristic.  
In order to treat characteristic zero and positive characteristic uniformly (without limiting ourselves to the split adjoint case) it is useful to use the language of schemes.  Let $\CG$ denote an affine group scheme of finite type over $\Z$
whose generic fiber $\CG_{\Q}$ is quasisimple.
By \cite[IV~11.1.1]{EGA}, there exists a finite subset $S\subset \Spec\Z$ such that $\CG$ is
flat over the complement of $S$.  By \cite[XIX~2.5]{SGA}, $\CG$ is a semisimple group scheme over
the complement of a (possibly larger) finite subset of $\Spec\Z$; 
in particular, the fibers are all semisimple algebraic groups.
By \cite[XXII~2.8]{SGA}, for every sufficiently large prime $p$, all the fibers $\CG_{\F_p}$ have the same root datum  (in particular, the same Dynkin diagram).
By \cite[IV~9.5.5]{EGA}, 
for every sufficiently large prime $p$ and every algebraically closed field $\mathbb{F}$  of  characteristic $p$,
we have $\delta_a^{\CG_{\C}}=\delta_a^{\CG_{\mathbb{F}}}$.

\begin{prop}
Let $\CG$ be an affine group scheme of finite type over $\Z$ whose generic fiber is a quasisimple algebraic group.  For $m \in \mathbb{N}$ let  $\delta_m^{\CG_{\C}}$ be the dimension of the variety of all elements of $\CG_\mathbb{C}$ of order dividing $m$. Let $T=T_{a,b,c}$ be a hyperbolic triangle group. Then the following assertions hold:
\begin{enumerate}[(i)]
\item If $\delta_a^{\CG_\C}+\delta_b^{\CG_\C} + \delta_c^{\CG_\C} < 2 \dim \CG_{\C}$ then there is no Zariski dense representation from $T$ to $\CG_\C$.
\item If $\delta_a^{\CG_\C}+\delta_b^{\CG_\C}+\delta_c^{\CG_\C}=2 \dim \CG_{\C}$ then there are only finitely many conjugacy classes of representations of $T$ onto Zariski dense subgroups of $\CG_\C$,
and for almost all primes $p$, $\CG(\F_{p^\ell})$ is a quotient of $T$ for only finitely many $\ell$.  
\end{enumerate}
\end{prop}

\begin{remk}
\begin{enumerate}[(a)]
\item We do not know in case (ii)  if there is necessarily  any such representation.
\item If $\delta_a^{\uG(\C)}+\delta_b^{\uG(\C)}+\delta_c^{\uG(\C)}> 2 \dim \uG$ then as we will show later in the paper, there is usually a nontrivial deformation space of Zariski dense representations of $T$ into $\uG(\mathbb{C})$, but we do not know if this is always the case.
\end{enumerate}
\end{remk}

\begin{proof}
The first part follows from Scott's formula \cite{Scott} in a similar way to Marion's argument in \cite[\S2]{Marionconj}. The second part follows from the argument used in the proof of Theorem \ref{t:marconj}: In this case all Zariski dense representations of $T$ to $\CG_\mathbb{C}$ (and also to $\CG_\mathbb{F}$ if ${\rm char}(\mathbb{F})$ is sufficiently large) are  locally rigid and hence there are only finitely many. 
\end{proof}

In order to put the open cases of Theorem \ref{t:main} in the context of the rigidity conjecture, we now classify hyperbolic triples of integers for a simple algebraic group $\uG(\mathbb{C})$ over $\mathbb{C}$ of adjoint type.

\begin{prop}
Let $\uG(\mathbb{C})$ be of adjoint type and $(a,b,c)$ be a hyperbolic triple of integers.
\begin{enumerate}[(i)]
\item There are no reducible hyperbolic triples of integers for $\uG(\mathbb{C})$, i.e. triples with $\delta_a^{\uG(\mathbb{C})}+\delta_b^{\uG(\mathbb{C})}+\delta_c^{\uG(\mathbb{C})}<2 \dim \uG$.
\item The pairs $(\uG(\mathbb{C}),(a,b,c))$ for which $(a,b,c)$ is rigid for $\uG(\mathbb{C})$, i.e. $\delta_a^{\uG(\mathbb{C})}+\delta_b^{\uG(\mathbb{C})}+\delta_c^{\uG(\mathbb{C})}=2 \dim \uG$,  are as in the following table:\\
\begin{center}
\begin{tabular}{|l|l|}
\hline
$\uG(\mathbb{C})$ & $(a,b,c)$\\
\hline
$A_1(\mathbb{C})$ & any\\
$A_2(\mathbb{C})$ & $(2,b,c)$\\
$A_3(\mathbb{C})$ & $(2,3,c)$\\
$A_4(\mathbb{C})$ & $(2,3,c)$\\
\hline
$C_2(\mathbb{C})$ & $(2,3,c)$, $(3,3,c)$\\
\hline
$G_2(\mathbb{C})$ & $(2,4,5)$, $(2,5,5)$\\
\hline
\end{tabular}
\end{center}
\end{enumerate}
\end{prop}

\begin{proof}
One can extend the argument given in  the proof of \cite[Theorem 3]{Marionconj} to simple algebraic groups defined over $\mathbb{C}$ and to the case where $a,b,c$ are not necessarily primes. Note that by model theory, the result over $\mathbb{C}$ follows from the same result over algebraically closed fields of large characteristic.
\end{proof}

\begin{rmk}
These cases are also rigid over an algebraically closed field of characteristic $p$, for every prime $p$. Thus, our proof of Marion's conjecture (see Theorem \ref{t:marconj}) shows that in all these cases $T_{a,b,c}$ is not saturated with finite quotients of type $X$ (where $X$ is one of the six types in the above table).
\end{rmk}

\section{Deformations and saturation with finite quotients}\label{s:posres}

In the previous section, we used deformation theory to prove that many finite simple groups are not quotients of a given triangle group $T_{a,b,c}$. In this section we will use this theory to show that many of them are.
Recall that, with finitely many exceptions $(X,p,\ell)$, for any given connected Dynkin diagram $X$ and any prime $p$, there exists a simply connected, quasisimple  algebraic group $\uG$ defined over $\F_p$ of prime characteristic $p$  such that for each $\ell$, $X(p^\ell)$ is the quotient of  $\uG(\F_{p^\ell})$ by its center.  We denote by $X(\C)$ the group of complex points of
the adjoint simple algebraic group of type $X$ over $\C$.

Let us recall Definition \ref{d:sat}, defining the notion of $T=T_{a,b,c}$ being saturated with finite (simple) quotients of type $X$. Clearly this definition makes sense for every finitely  generated group, not only
for triangle groups.

The key point of our method is the following theorem which is of independent interest:

\begin{thm}\label{t:key}
Let $\Gamma$ be a finitely presented group and  $X$ be a connected Dynkin diagram. The following conditions are equivalent:
\begin{enumerate}[(1)]
\item There exists a representation $\rho: \Gamma \rightarrow X(\mathbb{C})$ with Zariski dense image  such that $\rho$ is not locally rigid in $\Hom(\Gamma,X(\mathbb{C}))$.
\item The group $\Gamma$ is saturated with finite quotients of type $X$.
\item For infinitely many primes $p$, $\Gamma$ has infinitely many quotients of type $X(p^\ell)$
\item If $f_\Gamma(p) \in \mathbb{N}\cup \{\infty\}$ denotes the cardinality of $\{\ell \in \mathbb{N}: X(p^\ell)\ \textrm{is a quotient} \ of \ \Gamma\}$, then $$\limsup_{p\to \infty}  f_\Gamma(p)=\infty.$$
\item The number of epimorphisms $|{\rm Epi}(\Gamma,X(p))|$, up to conjugation by $X(p)$, is unbounded as a function of $p$.
\end{enumerate}
\end{thm}

We use the following lemma.

\begin{lem}
\label{specialize}
Let $A\subset B$ be integral domains finitely generated over $\Z$  such that $B$ is a finitely generated $A$-module.  For $q$ a prime power, let $f(q)$ denote
the cardinality of the set of ring homomorphisms $\phi\colon B\to \F_q$ such that $\phi(A) = \F_q$.  The following conditions are equivalent:
\begin{enumerate}[(a)]
\item The field of fractions of $B$ is an extension of $\Q$ of positive transcendence degree.
\item For infinitely many primes $p$, $f(p^\ell)>0$ for all $\ell\in\N$.
\item The restriction of $f$ to primes is unbounded.
\item There exists a number field $K$ with ring of integers $\cO$ such that if $p$ is sufficiently large
and $\Hom(\cO,\F_q)$ is not empty for $q=p^\ell$, then $f(q) > 0$.
\end{enumerate}

\end{lem}

\begin{proof}
The image of $\Spec A\to \Spec \Z$ is a constructible set, so either it is a single prime $(p)$ or
it is the complement of a finite set of primes.  In the first case, the field of fractions of $A$ has characteristic $p>0$
and $f(q)=0$ for all powers of every prime except $p$.  Thus, we may assume that we are in the second case.

Let $\CX=\Spec A$ and $\CY=\Spec B$.   Let $K_0$ denote the integral closure of $\Q$ in the fraction field of $B$, which contains the
integral closure of $\Q$ in the fraction field of $A$.
If $K$ is any finite extension of $K_0$ which is Galois over $\Q$,
then $\CX\times_{\Spec \Q} \Spec K$ (respectively, $\CY\times_{\Spec \Q} \Spec K$) is a finite disjoint
union of geometrically irreducible components \cite[Cor.~4.5.10]{EGA}.
Let $\CX_{\cO}:= \CX\times_{\Spec \Z}\Spec \cO$
and $\CY_{\cO}:= \CY\times_{\Spec \Z}\Spec \cO$, where $\cO$ is the ring of integers of $K$.
If $p$ is unramified in $K$ and decomposes into primes of $\cO$ with residue field isomorphic to $\F_{p^j}$,
then for every positive integer $i$, and every point $x$ of $\CX$ with residue field isomorphic to $\F_{p^i}$,
there are $\frac{[K:\Q](i,j)}{j}$ points of $\CX_{\cO}$ lying over $x$, and each has residue field isomorphic to
$\F_{p^{[i,j]}}$.  Likewise, for every point $y$ of $\CY$ with residue field isomorphic to $\F_{p^i}$,
there are $\frac{[K:\Q](i,j)}{j}$ points of $\CY_{\cO}$ lying over $y$, and each has residue field isomorphic to
$\F_{p^{[i,j]}}$.

As $A$ and $B$ are integral domains, the same is true of $A\otimes\Q$ and $B\otimes \Q$, so the
irreducible components of $A\otimes K$ and $B\otimes K$ each form a single Galois orbit.
If $n$ denotes the dimension of the generic fiber of $\CX$, then
every geometric component of this generic fiber has dimension $n$, and the same is true for $\CY$.
By definition of $K$, the generic fiber of $\CX$ (respectively, $\CY$) over $\Spec \cO$ has geometrically irreducible  components,  so the same is true of all components of all but finitely many
fibers of $\CX$ (respectively, $\CY$) over closed points of
$\Spec \cO$ \cite[Prop.~9.7.8]{EGA}.  Moreover, all components of all but finitely many primes are $n$-dimensional \cite[Prop.~9.5.5]{EGA}.

If $n=0$, therefore, for all $p$ sufficiently large, the sum $\sum_{l=1}^\infty f(p^\ell)$ is  finite
and bounded above by a constant independent of $p$.  Thus, none of the conditions (b)--(d) can hold, and we are done.
We therefore assume $n>0$.  If $p\gg 0$ and $q = p^\ell$ is the cardinality of a residue field of a prime ideal of $\cO$, then by the Weil bound,
$$|\CY_K(\F_q)| \ge \frac{q^n}2,$$
while if $\F_{p^k}$ is a subfield of $\F_q$,
$$|\CX(\F_{p^k})| = O(p^{nk}).$$
The fibers of the morphism $\CY_K\to \CX_K$ are finite and therefore bounded.
It follows that if $q\gg 0$, the number of points in $\CY(\F_q)$ which map to points of $\CX(\F_q)$ which are not defined over a proper subfield
can be bounded below by $\epsilon q^n$ for some $\epsilon > 0$.  This immediately implies (d), and combined with Chebotarev density implies (b) and (c).
\end{proof}

\begin{lem}
\label{annoying}
If $\uG$ is a quasisimple algebraic group over a field $K$ of characteristic zero, $F$ an extension field of $K$, and
$\Gamma\subset \uG(F)$ a Zariski dense subgroup such that $\Gamma\cap \uG(K)$ is of finite index in $\Gamma$, and the adjoint trace $\Ad(\gamma)$ lies in $K$ for all $\gamma\in \Gamma$, then there exists a finite extension $L$ of $K$ in $F$ such that $\Gamma\subset \uG(L)$.
\end{lem}

\begin{proof}
Let $\gamma\in \Gamma$.  As the adjoint representation of $\uG$ is irreducible and $\Gamma':=\Gamma\cap \uG(K)$ is Zariski dense in $\uG$, it follows that the $K$-span in the endomorphism algebra of the Lie algebra of $\uG$ of $\{\Ad(\gamma')\mid \gamma'\in\Gamma'\}$ is the whole matrix algebra.  For each $\gamma\in \Gamma$, the trace of
$\Ad(\gamma\gamma')$ lies in $K$ for all $\gamma'\in\Gamma'$, and it follows that
$\Ad(\gamma)$ is defined over $K$.  The adjoint representation factors through the adjoint group $\uG^{\mathrm{ad}}$, on which it is faithful, and it follows that the image of each $\gamma$
in $\uG^{\mathrm{ad}}(F)$ lies in $\uG^{\mathrm{ad}}(K)$.  Thus, each $\gamma$ lies in $\uG(K_\gamma)$
where $K_\gamma$ is a finite extension of $K$.  Choosing one  representative $\gamma$ for each class in $\Gamma/\Gamma'$, we obtain a  finite extension $L$ such that
$\Gamma\subset \uG(L)$.

\end{proof}

We can now prove  Theorem~\ref{t:key}.

\begin{proof}
Let $\CG /\Spec \Z$ be the simply connected split Chevalley group scheme associated to the Dynkin diagram $X$, and let $\uG:= \CG_{\C}$ denote its fiber over $\Spec \C$.
As $\Gamma$ is finitely presented, $H^2(\Gamma,A)$ is finite for any finite $\Z\Gamma$-module $A$.
Therefore, there are finitely many different central extensions of $\Gamma$ by a given finite abelian group.
Thus, whenever $X(p^\ell)$ is a quotient of $\Gamma$, $\CG(\F_{p^\ell})$ is a quotient of $\tilde \Gamma$ for $\tilde \Gamma$ an element of
a finite set $\Sigma$ of (finitely generated) central extensions of $\Gamma$, and conversely, if $\CG(\F_{p^\ell})$ is
a quotient of any $\tilde \Gamma\in \Sigma$ (indeed, of any central extension of $\Gamma)$, then $X(p^\ell)$ is a quotient of $\Gamma$.
Likewise, for every homomorphism $\Gamma\to X(\C)$ with Zariski dense image, there corresponds
a dense homomorphism from some $\tilde\Gamma\in\Sigma$ to  $\uG(\C)$, the simply connected cover of $X(\C)$, and conversely.
It therefore suffices to prove the equivalence of the following variants of conditions (1)--(5), applied to each $\tilde \Gamma\in \Sigma$:

\begin{enumerate}[(1)]
\item[($\tilde{1}$)] There exists a representation $\rho: \tilde \Gamma \rightarrow \uG(\mathbb{C})$
with Zariski dense image  such that $\rho$ is not locally rigid in $\Hom(\tilde \Gamma,\uG(\mathbb{C}))$.
\item[($\tilde{2}$)] The group $\tilde \Gamma$ is saturated with  quotients of type $\CG(\F_q)$, where $q$ ranges over prime powers.
\item[($\tilde{3}$)]  For infinitely many primes $p$, $\tilde \Gamma$ has infinitely many quotients of type $\CG(\F_{p^\ell})$.
\item[($\tilde{4}$)]  If $f_\Gamma(p) \in \mathbb{N}\cup \{\infty\}$ denotes the cardinality of $\{\ell \in \mathbb{N}: \CG(\F_{p^\ell})\ \textrm{is a quotient} \ of \ \tilde\Gamma\}$, then $$\limsup_{p\to \infty}  f_{\tilde\Gamma}(p)=\infty.$$
\item[($\tilde{5}$)]  The number of epimorphisms $|{\rm Epi}(\tilde \Gamma,\CG(\F_p))|$, up to conjugation by $\CG(\F_p)$, is unbounded as a function of $p$.
\end{enumerate}

Let $\CY := \Hom({\tilde\Gamma},\CG)$, and let $\uY := \CY_{\C}$
denote the fiber of $\CY$ over $\Spec \C$.
Then $\uY = \Spec R$ for some finitely generated $\C$-algebra $R$, and there exists a universal
homomorphism $\Phi\colon {\tilde\Gamma}\to \CG(R)$, such that
every homomorphism $\rho\colon {\tilde\Gamma}\to \uG(\C)$ corresponds, by specialization, to an element of $\uY(\C)$.

We now assume that condition ($\tilde{1}$) holds.
Let $\uZ$ denote an irreducible component of $\uY$ with function field $F$.  We assume that $\uZ$ is chosen
such that there exists $z\in \uZ(\C)$ corresponding to a homomorphism with Zariski dense image which is not locally rigid.
It follows that the representation
${\tilde\Gamma}\to \CG(F)$ corresponding to the generic point is also Zariski dense.
For each $\gamma\in {\tilde\Gamma}$, there exists an element $T_\gamma\in R$ given by $\Tr(\Ad(\Phi(\gamma)))$.  If every $T_\gamma$ maps to a constant under the homomorphism $R\to F$, then traces are fixed as $\rho$ varies over $\CZ(\C)$.  By the Brauer-Nesbitt theorem, the representations $\Ad\circ \rho$ are all isomorphic as $\rho$ ranges over all representations
$\tilde \Gamma\to \uG(\C)$ corresponding to points of $\uZ(\C)$.  As the outer automorphism group of $\uG$ is finite, up to conjugation, there are
only finitely many possibilities for $\rho$.  This contradicts the assumption that some $\rho$ is not locally rigid.  Thus, some $T_\gamma$ has non-constant image in $F$.

We now apply Weisfeiler's strong approximation theorem
to $\rho\colon {\tilde\Gamma}\to \CG(F)$.
Let $A_0$ denote the subring of $F$ generated by the image of
$\{T_\gamma\mid\gamma\in{\tilde\Gamma}\}$ under the natural homomorphism $R\to F$, and let $K$ denote the field of fractions of $A_0$.  By  \cite[Theorem 1.1]{Weisfeiler}, there exists a normal subgroup $\Gamma'$ of finite index in ${\tilde\Gamma}$, a finitely generated $\Z$-algebra
$A\subset K$ with fraction field $K$, a group scheme $\CH$ over $A$, an isomorphism $i\colon \CH_F\to \CG_F$, where
$$\CH_F := \CH\times_{\Spec A} \Spec F,\ \CG_F :=  \CG\times_{\Spec \C}\Spec F,$$
and a homomorphism $\Gamma'\to \CH(A)$ whose composition with $i$ coincides with the
restriction of $\rho$ to $\Gamma'$ and  such that $\Gamma'$ maps onto $\CH(\F_q)$
for every surjective homomorphism $A\to \F_q$.

For each coset of $\Gamma'$ in ${\tilde\Gamma}$, we choose a representative $\gamma_i$.  By Lemma~\ref{annoying}, $i^{-1}(\rho(\gamma_i))\in \CH(L_i)$
for some finite extension $L_i/K$.  We can choose a finite extension $L$ of $K$ such that each $L_i$ is contained in $L$ and  $\CH_L$ is split.
Let $B$ denote a finitely generated $A$-algebra in $L$ such that $L$ is the fraction field of $B$ and $i^{-1}(\rho(\gamma_i))\in \CH(B)$ for each $i$.
Thus $i^{-1}(\rho({\tilde\Gamma}))\in \CH(B)$.  Replacing $B$ by $B[1/b]$ for some nonzero $b\in B$, we may assume $\CH_B$ is split.  As $\Spec B\to\Spec A$ is
generically finite, after replacing $A$ and $B$ by $A[1/a]$ and $B[1/a]$ respectively, we may further assume that $B$ is module-finite over $A$.

We now apply Lemma~\ref{specialize}.  Condition (a) holds, so conditions (b)--(d) hold as well.  They imply conditions ($\tilde{2}$)--($\tilde{5}$).

We now assume, on the contrary, that every representation $\rho: {\tilde\Gamma} \rightarrow \uG(\mathbb{C})$ with Zariski dense image is locally rigid.
This implies that for each irreducible component $\uZ$ of $\uY$, either $\rho(\tilde\Gamma)$ fails to be Zariski dense for all $\rho$ parametrized by $z\in \uZ(\C)$
or for each $\gamma\in \tilde \Gamma$, $\Tr (\Ad(\rho(\gamma)))$ is constant on $\uZ(\C)$.  As $\uY$ has finitely many irreducible components, there are only
finitely many possibilities for the function $\gamma\mapsto \Tr (\Ad(\rho(\gamma)))$ as $\rho$ ranges over all Zariski dense homomorphisms $\rho\colon \tilde \Gamma\to \uG(\C)$.
The set of such functions is stable by the automorphism group of $\C$, so it follows that all traces of all Zariski dense homomorphisms lie in some number field $K$.

We claim that there exists a finite collection $\CY_i$ of locally closed subschemes of $\CY$, each smooth over $\Spec\Z$, such that
every closed point of $\CY$ of characteristic $p$ sufficiently large lies in some $\CY_i$.  Indeed, we note first that without loss of generality, we may assume that $\CY$ is
irreducible and reduced.    The generic fiber of $\CY$ is a variety over $\Q$, so its singular locus is a proper closed subvariety.
Let $\CZ$ denote the Zariski closure of this subvariety in $\CY$, endowed with its reduced induced subscheme structure.
By Noetherian induction, we may assume that $\CZ$ admits such a finite collection, so it suffices to prove the same for
$\CY\setminus \CZ$.  The generic fiber of this scheme is nonsingular, so it is smooth over some open neighborhood of the generic point of
$\Spec\Z$ \cite[17.7.11]{EGA}.  At the cost of throwing out a finite set of closed fibers, what remains is smooth.

If $y$ is a closed point of $\CY$ with residue field $\F_q$ whose characteristic $p$ is sufficiently large,
by the infinitesimal lifting property of smooth morphisms, there exists a morphism
from the spectrum of $W(\F_q)$, the ring of Witt vectors over $\F_q$, to $\CY$,  mapping the closed point of
$\Spec W(\F_q)$ to $y$.
If $y$ corresponds to a surjective homomorphism $\phi\colon {\tilde\Gamma}\to \CG(\F_q)$, then
we have a homomorphism $\tilde{\phi}\colon {\tilde\Gamma}\to \CG(W(\F_q))$ which lifts this surjective homomorphism.  By
a theorem of  Vasiu \cite{Vasiu}, if $q$ is sufficiently large, this implies that $\tilde{\phi}$ has dense image.
Let $\Q_q$ denote the fraction field of $W(\F_q)$ (i.e., the unramified extension of $\Q_p$ of degree $[\F_q:\F_p]$).
As there are finitely many subextensions of $\Q_p$ in $\Q_q$, there exists $\gamma\in \tilde \Gamma$ such that
$$\Q_p(\Tr(\Ad(\tilde{\phi}(\gamma)))) = \Q_q.$$
If  $[\F_q:\F_p] > [K:\Q]$, this is impossible,
which shows that conditions ($\tilde{2}$)--($\tilde{4}$) of the theorem cannot hold.

For condition ($\tilde{5}$), we note that if $p$ is sufficiently large, $\Ad$ is an irreducible representation of $\CG({\F_p})$, so
two surjective homomorphisms
$\phi_1,\phi_2\colon {\tilde\Gamma}\to \CG(\F_p)$ such that $\Ad\circ \phi_1$ and $\Ad\circ\phi_2$
have the same semisimplification are equivalent up to tensor product with a character of the center of
$\CG(\F_q)$ (whose order is bounded independent of $p$).
If $\Ad\circ \phi_1$ and $\Ad\circ\phi_2$ have distinct semisimplifications, then
there exists $\gamma\in {\tilde\Gamma}$ such that $\Tr(\Ad(\phi_1(\gamma)))\neq \Tr(\Ad(\phi_2(\gamma)))$.
This implies that
$$\Tr(\Ad(\tilde\phi_1(\gamma)))\neq \Tr(\Ad(\tilde\phi_2(\gamma)))$$
and therefore that $\tilde \phi_1$ and $\tilde \phi_2$ correspond to non-conjugate Zariski dense homomorphisms $\tilde \Gamma\to \CG(\bar \Q_p)$.
Fixing an isomorphism between $\C$ and $\bar\Q_p$, this gives an upper bound on the number of conjugacy classes of surjective
homomorphisms ${\tilde\Gamma}\to G(\F_p)$, for any prime $p$ sufficiently large
and therefore shows condition ($\tilde{5}$) cannot hold.
\end{proof}

\begin{rmk}
When $\Gamma=T$ is a Fuchsian group, and in particular a hyperbolic triangle group, by Weil \cite{Weil}, if $\rho: T\rightarrow X(\mathbb{C})$ is a representation with Zariski dense image then $\rho$ is not locally rigid if and only if $H^1(T,{\rm Ad}\circ \rho)\neq 0$ where ${\rm Ad}$ is the action of $X(\mathbb{C})$ on its Lie algebra. So, in this case, Theorem \ref{t:key}(1)  is equivalent to:
$$(1') \ \textrm{There exists a Zariski dense representation} \ \rho:T\rightarrow X(\mathbb{C}) \ \textrm{with}\ H^1(T,{\rm Ad}\circ\rho)\neq 0.$$
\end{rmk}

The next section (and to a large extent the remainder of the paper) will be devoted to showing that for most pairs $(T,X)$ we do, indeed, have $\rho$ satisfying condition $(1')$.
In fact, we prove the existence of a Zariski dense homomorphism, not locally rigid, from $T$ to the real points of the compact real simple (in particular, adjoint) group with Dynkin diagram $X$.

\section{Deformations of the principal homomorphism}\label{s:dph}

In  this section we will apply Theorem \ref{t:key} to show that many triangle groups are saturated with finite  quotients of various types. To this end we will deform the  principal homomorphism introduced in \S \ref{s:dtfg}.
We have the following theorem:

\begin{thm}\label{t:ll}  Let $T=T_{a,b,c}$, $G=\uG(\mathbb{R})$,  where $\uG$ is a compact,  adjoint, simple, real algebraic group,  $\rho_0: T \rightarrow G$ a homomorphism, and $\uH$ the Zariski closure of $\rho_0(T)$. Assume
\begin{enumerate}[(a)]
\item  $\uH$ is semisimple and connected.
\item $\uH$ is a maximal proper subgroup of $\uG$.
\item If $\mathfrak{g}
$ (respectively, $\mathfrak{h}$) is the Lie algebra of $\uG$ (respectively, $\uH$) (where the action is via ${\rm Ad}\circ \rho_0$), then
$$\dim H^1(T, \mathfrak{h}) < \dim H^1(T,\mathfrak{g}).$$
\end{enumerate}
Then  the following assertions hold:
\begin{enumerate}[(i)]
\item The homomorphism $\rho_0$ determines  a  nonsingular point of $\Hom(T,\uG)$.
\item Some irreducible component of $\Hom(T,\uG)$  containing $\rho_0$ has dimension  $\dim H^1(T, \mathfrak{g})+\dim \uG$ and contains a nonsingular point $\rho$ with Zariski dense image.
\item For $\rho$ as in (ii), $\dim H^1(T, {\rm Ad}\circ \rho)=\dim H^1(T, {\rm Ad}\circ \rho_0)$.
\end{enumerate}
\end{thm}

\begin{proof}
Since the Zariski closure $\uH$ of $\rho_0(T)$ is a maximal subgroup of $\uG$ we have  $Z_{\uG}(\uH)=Z(\uH)$, which is finite since $\uH$ is semisimple. Hence ${\rm Ad}\circ \rho_0$ has no invariants on $\mathfrak{g}$. Since the adjoint representation of a simple group in characteristic zero is self-dual, $({\rm Ad} \circ \rho_0)^*$ has no invariants on $\mathfrak{g}^*$. In particular, Theorem \ref{t:weil} now yields the first part.
Moreover, it now follows from (\ref{e:dimptildeone}) and (\ref{e:dimpone})
that $$ \dim Z^1(T,\mathfrak{h})-\dim \uH=\dim H^1(T,\mathfrak{h}) \quad {\rm and} \quad \dim Z^1(T,\mathfrak{g})-\dim \uG=\dim H^1(T,\mathfrak{g})$$ and  \cite[Corollary 2.5]{LL} now yields the second part.
Let us now consider the final part. Being simple, $\uG$ has finite center. Since  $\rho: T\rightarrow \uG(\R)$ is a Zariski dense representation, it follows that ${\rm Ad}\circ \rho$ has no invariants on $\mathfrak{g}$. The same  argument as the one given above for $(\rm Ad\circ \rho_0)^*$ yields that $({\rm Ad}\circ \rho)^*$ has no invariants on $\mathfrak{g}^*$. Since $\rho_0$ and $\rho$ lie in a common irreducible component of $\Hom(T,\uG)$,   the result now follows from Lemma \ref{l:h1equal}.
\end{proof}

We will start with $\rho_0: T\rightarrow \uG(\R)$ induced from the principal homomorphism  ${\rm SO}(3) \rightarrow \uG$, where $$T \not \in S=\{T_{2,4,6},T_{2,6,6},T_{2,6,10},T_{3,4,4},T_{3,6,6},T_{4,6,12}\}$$ so that $T$ is ${\rm SO}(3)$-dense (see (\ref{e:trianglenotso3dense})). In the notation of Theorem \ref{t:ll}, $\uH \cong {\rm SO}(3)$ and condition (a) is satisfied. In many cases, condition (b) is also satisfied by Dynkin's classification of maximal subalgebras of simple Lie algebras (see \cite{D2,D3}).
To avoid confusion we will sometimes write
$\rho_0^{\uG}$ instead of $\rho_0$.
If $\mathfrak{g}$ denotes the Lie algebra of $\uG$, we will for conciseness  write $\mathfrak{g}$ for $({\rm Ad} \circ \rho_0^{\uG}\vert_{\mathfrak{g}})$, i.e. the action of $T$ on $\mathfrak{g}$ via ${\rm Ad}\circ\rho_0^{\uG}$.

\begin{prop}\label{p:dynkin}
Let $\uG$ be a compact adjoint simple group over $\mathbb{R}$ of type  $A_2$, $B_n$ ($n \geq 4$), $C_n$ ($n \geq 2$), $E_7$, $E_8$, $F_4$ and $G_2$. Then the image of a principal homomorphism from ${\rm SO}(3)$ into $\uG$ is a maximal subgroup.
\end{prop}

Note that in the case where $\uG$ is of type $A_1$ (i.e. $\uG={\rm SO}(3)$), we saw that the  representation to ${\rm SO}(3)$ is locally rigid and indeed in this case we have: for any hyperbolic triangle group $T=T_{a,b,c}$ and any fixed prime $p$, $T$ has only finitely many quotients of the form ${\rm PSL}_2(p^\ell)$. On the other hand:

\begin{thm}\label{t:exrep}
Let $T=T_{a,b,c}$ be a hyperbolic triangle group not in $S$ (see (\ref{e:trianglenotso3dense}))
 and $\uG$ be a compact, adjoint, real simple  group of one of the following types:
$$A_2,\  B_n (n\geq 4), \ C_n (n \geq 2), \ E_7, E_8, F_4, G_2.$$
Let $\rho_0: T \rightarrow  {\rm SO}(3,\R) \rightarrow \uG(\R)$ be the representation of $T$ induced from the principal homomorphism  ${\rm SO}(3) \rightarrow \uG$.
Unless  $(T,X)$ is as in Table \ref{tab:exrep} (see \S\ref{s:tab}), there exists  a nonsingular $\R$-point $\rho_1\in \Hom(T,\uG)$, with Zariski dense image, which belongs to the same  irreducible component of $\Hom(T,\uG)$ as $\rho_0$  and
which satisfies
$$\dim H^1(T, {\rm Ad}\circ \rho_1\vert_{\mathfrak{g}})=\dim H^1(T, \mathfrak{g})>0.$$
In particular, $T$ is saturated with finite  quotients of type $X$,  unless $(T,X)$ is as  in Table \ref{tab:exrep}.
\end{thm}

\begin{rmk}\label{rmk:exrep}
Given $X$, the triples $(a,b,c)$ appearing in Table \ref{tab:exrep} appear also as  rigid triples  (in Marion's sense) for an adjoint algebraic group $\uG$ of type $X$, over an algebraically closed field of prime characteristic $p$, for every prime $p$.
\end{rmk}

\begin{proof}
 Let $\uH$ be the Zariski closure of $\rho_0(T)$. Then $\uH={\rm SO}(3)$ is simple and connected, and by Proposition \ref{p:dynkin}, $\uH$ is a maximal subgroup of $\uG$. Clearly $\rho_0^{\uG}=\rho_0^{\uH}$ (indeed here $\uH={\rm SO}(3)$), and by  Lemma \ref{l:ineq}, $\dim H^1(T,\mathfrak{h})=0$ and $\dim H^1(T,\mathfrak{g})>0$ unless $(T,X)$ is as in Table \ref{tab:exrep}. The result now follows immediately from  Theorems \ref{t:ll} and \ref{t:key}.
\end{proof}

Most of the cases where $\uG$ is of type $A_r$ or $D_r$ can also be treated using Theorem \ref{t:ll},  not directly from the principal homomorphism from ${\rm SO}(3)$ but rather via a ``two-step ladder". More precisely:

\begin{thm}\label{t:twosteps}
Let $T=T_{a,b,c}$ be a hyperbolic triangle group not in $S$ (see (\ref{e:trianglenotso3dense})) and $\uG$ be a compact, adjoint, real simple group of type $X=A_r$ (with $r\geq 3$ and $r\neq 6$), or $X=D_r$  (with $r \geq 5$). Let $\uH$ be a closed subgroup of $\uG$ of type $Y$, as follows:
\begin{enumerate}[(a)]
\item If $X=A_r$, let $Y=B_{r/2}$ if $r$ is even, otherwise $Y=C_{(r+1)/2}$.
\item If $X=D_r$, let $Y=B_{r-1}$.
\end{enumerate}
Let $\rho_1: T\rightarrow \uH(\R)$ be the nonsingular, Zariski dense representation provided by Theorem \ref{t:exrep} (excluding the cases of Table \ref{tab:exrep}). Then the following assertions hold:
\begin{enumerate}[(i)]
 \item If $\dim H^1(T,\mathfrak{h}) < \dim H^1(T,\mathfrak{g})$ then there exists a nonsingular representation  $\rho_2: T \rightarrow \uG(\R)$ in the  irreducible component of $\Hom(T,\uG)$ containing $\rho_1$, with Zariski dense image and satisfying
$$\dim H^1(T,{\rm Ad} \circ \rho_2\vert_{\mathfrak{g}})=\dim H^1(T,{\rm Ad}\circ \rho_1\vert_{\mathfrak{g}})=\dim H^1(T,\mathfrak{g})>0.$$
\item The cases for which $\dim H^1(T,\mathfrak{h})=\dim H^1(T,\mathfrak{g})$ are described in Table \ref{tab:twosteps} (see \S\ref{s:tab}).
\item In particular, $T$ is saturated with finite  quotients of type $X$, except possibly if $X\in\{(A_3,A_4)\}$ and $(a,b)\in\{(2,3),(3,3)\}$  or $(T,X)$  appears in Table  \ref{tab:twosteps}.
\end{enumerate}
\end{thm}

\begin{rmk}
For the moment, we exclude the cases $X=A_6$ and $X=D_4$, as  these cases require $Y=B_3$, which is not covered by Theorem \ref{t:exrep}.
\end{rmk}

\begin{proof}
 Note that $\rho_0^{\uG}=\rho_0^{\uH}$ (see \cite[Theorem B]{SS}) and so, by Theorem \ref{t:exrep}, $\rho_1$ and $\rho_0^{\uG}$ are in a common irreducible component of $\Hom(T,\uH)$. Hence $\rho_1$ and $\rho_0^{\uG}$ are in a common irreducible component of $\Hom(T,\uG)$. Furthermore, $\uH$ is semisimple and connected as well as maximal in $\uG$.
Hence Theorem \ref{t:ll} yields the existence of a nonsingular representation $\rho_2: T \rightarrow \uG(\R)$  in the same irreducible component of $\Hom(T,\uG)$ as $\rho_1$, with Zariski dense image and  satisfying $$\dim H^1(T,{\rm Ad} \circ \rho_2\vert_{\mathfrak{g}})=\dim H^1(T,{\rm Ad}\circ \rho_1\vert_{\mathfrak{g}}),$$
provided that $\dim H^1(T,{\rm Ad}\circ \rho_1\vert_\mathfrak{h})<\dim H^1(T,{\rm Ad}\circ \rho_1\vert_\mathfrak{g})$. Now by Theorem \ref{t:exrep},
$$\dim H^1(T,{\rm Ad}\circ \rho_1\vert_\mathfrak{h})=\dim H^1(T,{\rm Ad}\circ \rho_0^{\uH})=\dim H^1(T,\mathfrak{h}).$$
Also ${\rm Ad}\circ \rho_1$ has no invariants on $\mathfrak{g}$ (since $\overline{\rho_1(T)}=\uH$ is a maximal subgroup with finite center of the simple group $\uG$) and  ${\rm Ad\circ \rho_0^{\uG}}$ has no invariants on $\mathfrak{g}$ (since $\rho_0^{\uG}$ is the representation induced from the principal homomorphism). Furthermore, as the adjoint representation of a simple group in characteristic zero is self-dual, $({\rm Ad}\circ \rho_1)^*$ and $({\rm Ad}\circ \rho_0^{\uG})^*$ have no invariants on $\mathfrak{g}^*$. Since $\rho_1$ and  $\rho_0^{\uG}$ lie in a common irreducible component of $\Hom(T,\uG)$, Lemma \ref{l:h1equal} now yields
$$ \dim H^1(T,{\rm Ad}\circ \rho_1\vert_{\mathfrak{g}})=\dim H^1(T,{\rm Ad}\circ \rho_0^{\uG}\vert_{\mathfrak{g}})=\dim H^1(T,\mathfrak{g}).$$
We therefore need to prove that except for $T$ and $X$ as in Table \ref{tab:twosteps}, we have
\begin{equation}\label{e:ineqclaim} \dim H^1(T,\mathfrak{g})>\dim H^1(T,\mathfrak{h}),
\end{equation}
 and then the result  will follow from Theorem \ref{t:key}.

The following line of argument will also be repeated  in the proofs of Theorems \ref{t:twostepssmall} and \ref{t:threestepssmall}, so let us isolate it here:
\begin{cla}\label{c:cla}
Observe that as $\uH$ is a subgroup of $\uG$, if inequality (\ref{e:ineqclaim}) does not hold, then equality holds instead.   For convenience, we let $n_1=a$, $n_2=b$ and $n_3=c$.
Noting, by (\ref{e:exp}), that the exponents of $\uH$ form a subset of the set of exponents of $\uG$, we let $E$ be the set of exponents of $\uG$ which are not exponents of $\uH$.    It  follows from (\ref{e:dimge}) that
$$\dim \uG-\dim \uH= \sum_{e \in E}(1+2e).$$ Proposition \ref{p:dimpone} now yields
$$\dim H^1(T, \mathfrak{g})-\dim H^1(T,\mathfrak{h})=\sum_{e \in E}(1+2e)-\sum_{k=1}^3\sum_{e \in E}\left(1+2\left\lfloor \frac{e}{n_k} \right\rfloor \right).$$ In particular, $\dim H^1(T, \mathfrak{g})>\dim H^1(T,\mathfrak{h})$ if and only if
\begin{equation}\label{e:crucial}
\sum_{k=1}^3\sum_{e \in E} \left \lfloor  \frac{e}{n_k}\right \rfloor< \sum_{e \in E}e-|E|.
\end{equation}
We let $L_{(n_1,n_2,n_3),\uH,\uG}$ and $R_{\uH,\uG}$ denote respectively the LHS and the RHS of (\ref{e:crucial}).
In order to check whether  inequality (\ref{e:crucial}) holds or not, it is useful to put a partial order on the set of  hyperbolic triples as follows. Given two hyperbolic triples $(n_1,n_2,n_3)$ and $(n_1',n_2',n_3')$,  we say that $(n_1,n_2,n_3) \leq (n_1',n_2',n_3')$ if and only if $n_1 \leq n_1'$, $n_2\leq n_2'$ and $n_3 \leq n_3'$.   We note that  $L_{(n_1,n_2,n_3),\uH,\uG}$ decreases with respect to this partial order. Thus for a fixed pair $(\uG,\uH)$, if (\ref{e:crucial}) holds for a triple $(n_1,n_2,n_3)$, it holds also for every greater triple.
\end{cla}

Among all hyperbolic triples, exactly three are minimal: $(2,3,7)$, $(2,4,5)$ and $(3,3,4)$.  They are the starting points of this line of argument.

Let us first treat the case where $X=A_r$. We let $s = \lfloor r/2\rfloor$. Note that $E=\{2,4,\dots,2s-2,2s\}$ and $|E|=s$. In particular $R_{\uH,\uG}$ is equal to:
$$R_r=2\sum_{j=1}^s j-s=s(s+1)-s=s^2$$ and $L_{(n_1,n_2,n_3),\uH,\uG}$ is equal to $$L_{r,(n_1,n_2,n_3)}=\sum_{k=1}^3\sum_{j=1}^s \left\lfloor \frac{2j}{n_k}\right\rfloor.$$ Also observe that
$L_{(2s,(n_1,n_2,n_3))}=L_{(2s+1,(n_1,n_2,n_3))}$ and $R_{2s}=R_{2s+1}$.
Given a fixed hyperbolic triple $(n_1,n_2,n_3)$, we first claim that
\begin{equation}\label{e:claim}
L_{(r+2,(n_1,n_2,n_3))}-L_{(r,(n_1,n_2,n_3))} \leq R_{r+2}-R_r.
\end{equation}
Given a fixed hyperbolic triple $(n_1,n_2,n_3)$, this latter claim shows that if inequality (\ref{e:crucial}) holds for some $A_{r'}$ then it holds for all $A_r$, with $r \geq r'$.
Let us prove the claim.  We have $R_{r+2}-R_r=(s+1)^2-s^2=2s+1$ and, letting $\mu=1/n_1+1/n_2+1/n_3$, we get
\begin{eqnarray*}
L_{(r+2,(n_1,n_2,n_3))}-L_{(r,(n_1,n_2,n_3))}& =& \left\lfloor \frac{2(s+1)}{n_1} \right\rfloor+\left\lfloor \frac{2(s+1)}{n_2} \right\rfloor+\left\lfloor \frac{2(s+1)}{n_3} \right\rfloor\\
&\leq & 2\mu(s+1)\\
& \leq& \frac{82}{42}(s+1)
\end{eqnarray*}
as $\mu \leq 41/42$.

Note also that
$$ L_{(r+2,(n_1,n_2,n_3))}-L_{(r,(n_1,n_2,n_3))} \leq \alpha$$
where
\begin{multline*}
\alpha={\rm max}\Bigl(\left\lfloor\frac{2(s+1)}{n_1}\right\rfloor+ \left\lfloor\frac{2(s+1)}{n_2}\right\rfloor+ \left\lfloor\frac{2(s+1)}{n_3}\right\rfloor: \\
	(n_1,n_2,n_3) \in \{(2,3,7),(2,4,5), (3,3,4)\}\Bigr).
\end{multline*}

Suppose $s\geq 20$ (i.e. $r\geq 40$). Then $82(s+1)/42  \leq (2s+1)$ and the claim made in (\ref{e:claim}) follows in this case. For $s<20$ one directly checks that $\alpha\leq 2s+1$, fully establishing  (\ref{e:claim}) and the claim.

We now check whether inequality (\ref{e:crucial}) holds by a case by case analysis using the partial order defined above: we start with the triple $(2,3,7)$ proceeding with triples above it. We then repeat this procedure for the other two minimal triples, namely $(2,4,5)$ and $(3,3,4)$. In each branch of the partial ordered set, we check till we succeed; i.e. once inequality (\ref{e:crucial}) holds for a fixed $X=A_r$ and a given triple $(n_1,n_2,n_3)$, it also holds for any triple above it. We omit the details, which can be laboriously checked.
This finishes the proof of the theorem when $X=A_r$.

Let us now  treat the case where $X=D_r$.  Here $E=\{r-1\}$ and $|E|=1$. In particular $R_{\uH,\uG}$ is equal to $$R_r=r-2$$ and $L_{(n_1,n_2,n_3),\uH,\uG}$ is equal to $$L_r=\left\lfloor \frac{r-1}{n_1} \right\rfloor+ \left\lfloor \frac{r-1}{n_2} \right\rfloor+\left\lfloor \frac{r-1}{n_3} \right\rfloor.$$
One can give the following crude upper bound for $L_r$.
\begin{equation*}
L_r  \leq  \frac{r-1}{n_1}+\frac{r-1}{n_2}+\frac{r-1}{n_3}.
\end{equation*}
or the following one using (\ref{e:hypub})
\begin{equation*}
L_r \leq \frac{41(r-1)}{42}.
\end{equation*}
Since  $41(r-1)/42<r-2$ for $r>43$, we are now reduced to the case where $r \leq 43$. Here again we check inequality (\ref{e:crucial}) for $X=D_r$ ($r \leq 43$) by the same argument as before (i.e. going along the partial ordered set), this boils  down to a tedious finite case by case analysis.
\end{proof}

Here are more cases where a ``two-step ladder" works:

\begin{thm}\label{t:twostepssmall}
Let $T=T_{a,b,c}$ be a hyperbolic triangle group not in $S$ (see \ref{e:trianglenotso3dense})  and $\uG$  be a compact, adjoint, real simple group of type $X=B_3$ or type $X=E_6$.
 Let $\uH$ be the following subgroup of $\uG$ of type $Y$:
\begin{enumerate}[(a)]
\item If $X=B_3$, let $Y=G_2$.
\item If $X=E_6$, let $Y=F_4$.
\end{enumerate}
Let $\rho_1: T\rightarrow \uH(\R)$ be the nonsingular, Zariski dense representation provided by Theorem \ref{t:exrep} (excluding the cases of Table \ref{tab:exrep}). Then the following assertions hold:
\begin{enumerate}[(i)]
 \item If $\dim H^1(T,\mathfrak{h}) < \dim H^1(T,\mathfrak{g})$ then there exists a nonsingular representation $\rho_2: T\rightarrow \uG(\R)$ in the irreducible component of $\Hom(T,\uG)$ containing $\rho_1$, with Zariski dense image and satisfying
$$\dim H^1(T,{\rm Ad} \circ \rho_2\vert_{\mathfrak{g}})=\dim H^1(T,{\rm Ad}\circ \rho_1\vert_{\mathfrak{g}})=\dim H^1(T,\mathfrak{g})>0.$$
\item The cases for which $\dim H^1(T,\mathfrak{h})=\dim H^1(T,\mathfrak{g})$ are described in Table \ref{tab:twostepssmall} (see \S\ref{s:tab}).
\item In particular, $T$ is saturated with finite  quotients of type $X$, except possibly if  $X=B_3$ and $(a,b,c)\in\{(2,4,5),(2,5,5)\}$ or $(a,b)\in\{(2,3),(3,3)\}$, or $X=E_6$ and $$(a,b,c) \in \{(2,3,7),(2,3,8),(2,4,5),(2,4,7),(2,4,8)\}.$$
\end{enumerate}
\end{thm}

\begin{proof}
As $\rho_0^{\uG}=\rho_0^{\uH}$ (see \cite[Theorems A and B]{SS}), by Theorem \ref{t:exrep}, $\rho_1$ and $\rho_0^{\uG}$ are in a common irreducible component of $\Hom(T,\uH)$. Hence $\rho_1$ and $\rho_0^{\uG}$ are in a common irreducible component of $\Hom(T,\uG)$. Furthermore, $\uH$ is semisimple and connected as well as maximal in $\uG$.
Hence Theorem \ref{t:ll} yields the existence of a nonsingular representation $\rho_2: T \rightarrow \uG(\R)$  in the same irreducible component of $\Hom(T,\uG)$ as $\rho_1$, with Zariski dense image and  satisfying $$\dim H^1(T,{\rm Ad} \circ \rho_2\vert_{\mathfrak{g}})=\dim H^1(T,{\rm Ad}\circ \rho_1\vert_{\mathfrak{g}})$$  provided that $\dim H^1(T,{\rm Ad}\circ \rho_1\vert_\mathfrak{h})<\dim H^1(T,{\rm Ad}\circ \rho_1\vert_\mathfrak{g})$.
Now by Theorem \ref{t:exrep},
$$\dim H^1(T,{\rm Ad}\circ \rho_1\vert_\mathfrak{h})=\dim H^1(T,{\rm Ad}\circ \rho_0^{\uH})=\dim H^1(T,\mathfrak{h}).$$
Also ${\rm Ad}\circ \rho_1$ and ${\rm Ad}\circ \rho_0^{\uG}$ (respectively,  $({\rm Ad}\circ \rho_1)^*$ and $({\rm Ad}\circ \rho_0^{\uG})^*$) have no invariants on $\mathfrak{g}$ (respectively, $\mathfrak{g}^*$). Since $\rho_1$ and  $\rho_0^{\uG}$ lie in a common irreducible component of $\Hom(T,\uG)$, Lemma \ref{l:h1equal} now yields
$$ \dim H^1(T,{\rm Ad}\circ \rho_1\vert_{\mathfrak{g}})=\dim H^1(T,{\rm Ad}\circ \rho_0^{\uG}\vert_{\mathfrak{g}})=\dim H^1(T,\mathfrak{g}).$$

We therefore need to prove that except for $T$ and $X$ as in Table \ref{tab:twostepssmall}, we have $$ \dim H^1(T,\mathfrak{g})>\dim H^1(T,\mathfrak{h}),$$  and then the result  will follow from Theorem \ref{t:key}.
Noting that the exponents of $\uH$ form a subset of the exponents of $\uG$ (see (\ref{e:exp})), we now argue as in Claim \ref{c:cla} in the proof of Theorem \ref{t:twosteps}, adapting it  and its notation to our two cases.

Let us first treat the case where $X=B_3$ so that $Y=G_2$. By (\ref{e:exp}) $E=\{3\}$ and so $R_{G_2,B_3}=2$, while $L_{(n_1,n_2,n_3),G_2,B_3}=\sum_{k=1}^3 \lfloor 3/n_k \rfloor$. One now easily checks that inequality (\ref{e:crucial}) holds unless $n_2=3$, and this gives the first two lines in Table \ref{tab:twostepssmall}.

Let us now treat the case where $X=E_6$ so that $Y=F_4$. By (\ref{e:exp}) $E=\{4,8\}$ and so $R_{F_4,E_6}=10$, while $L_{(n_1,n_2,n_3),F_4,E_6}=\sum_{k=1}^3 (\lfloor 4/n_k \rfloor+\lfloor 8/n_k \rfloor)$. One now easily checks that inequality (\ref{e:crucial}) holds unless
$$(n_1,n_2,n_3) \in \{(2,3,7),(2,3,8),(2,4,5),(2,4,7),(2,4,8)\}.$$
\end{proof}

When $X$ is of type $D_4$ or $A_6$, we need to use a three-step ladder:

\begin{thm}\label{t:threestepssmall}
Let $T=T_{a,b,c}$ be a hyperbolic triangle group not in $S$ (see (\ref{e:trianglenotso3dense})) and $\uG$ be a compact, adjoint, real simple group of type $X=A_6$ or $D_4$. Consider the chain  $\uK <\uH$ of subgroups of $\uG$ where $\uK$ and $\uH$ are compact real forms of $G_2$ and $B_3$, respectively, inside $\uG$.
Let $\rho_2: T\rightarrow \uH(\R)$ be the nonsingular, Zariski dense representation provided by Theorem \ref{t:twostepssmall} (excluding the cases of Tables \ref{tab:exrep} and \ref{tab:twostepssmall}). Then the following assertions hold:
\begin{enumerate}[(i)]
\item $\dim H^1(T,\mathfrak{h})<\dim H^1(T,\mathfrak{g})$ (excluding the cases  of Tables \ref{tab:exrep} and \ref{tab:twostepssmall}).
\item There exists a nonsingular representation $\rho_3: T\rightarrow \uG(\R)$ in the  irreducible component of $\Hom(T,\uG)$ containing $\rho_2$, with Zariski dense image, and satisfying
$$\dim H^1(T,{\rm Ad} \circ \rho_3\vert_{\mathfrak{g}})=\dim H^1(T,{\rm Ad}\circ \rho_2\vert_{\mathfrak{g}})=\dim H^1(T,\mathfrak{g})>0.$$
\item In particular, $T$ is saturated with finite quotients of type $X$, except possibly if   $(a,b,c)\in\{(2,4,5),(2,5,5)\}$ or $(a,b)\in\{(2,3),(3,3)\}$.
\end{enumerate}
\end{thm}

\begin{proof}
As $\rho_0^{\uG}=\rho_0^{\uH}$ (see \cite[Theorem B]{SS}), by Theorem \ref{t:twostepssmall}, $\rho_2$ and $\rho_0^{\uG}$ are in a common irreducible component of $\Hom(T,\uH)$. Hence $\rho_2$ and $\rho_0^{\uG}$ are in a common irreducible component of $\Hom(T,\uG)$. Furthermore, $\uH$ is semisimple and connected as well as maximal in $\uG$.
Hence Theorem \ref{t:ll} yields the existence of a nonsingular representation $\rho_3: T \rightarrow \uG(\R)$ in the same irreducible component of $\Hom(T,\uG)$ as $\rho_2$, with Zariski dense image and  satisfying
$$\dim H^1(T,{\rm Ad} \circ \rho_3\vert_{\mathfrak{g}})=\dim H^1(T,{\rm Ad}\circ \rho_2\vert_{\mathfrak{g}})$$  provided that $\dim H^1(T,{\rm Ad}\circ \rho_2\vert_\mathfrak{h})<\dim H^1(T,{\rm Ad}\circ \rho_2\vert_\mathfrak{g})$.
Now by Theorem \ref{t:twostepssmall},
$$\dim H^1(T,{\rm Ad}\circ \rho_2\vert_\mathfrak{h})=\dim H^1(T,{\rm Ad}\circ \rho_0^{\uH})=\dim H^1(T,\mathfrak{h}).$$
Also ${\rm Ad}\circ \rho_2$ and ${\rm Ad}\circ \rho_0^{\uG}$ (respectively,  $({\rm Ad}\circ \rho_2)^*$ and $({\rm Ad}\circ \rho_0^{\uG})^*$) have no invariants on $\mathfrak{g}$ (respectively, $\mathfrak{g}^*$). Since $\rho_2$ and  $\rho_0^{\uG}$ lie in a common irreducible component of $\Hom(T,\uG)$, Lemma \ref{l:h1equal} now yields
$$ \dim H^1(T,{\rm Ad}\circ \rho_2\vert_{\mathfrak{g}})=\dim H^1(T,{\rm Ad}\circ \rho_0^{\uG}\vert_{\mathfrak{g}})=\dim H^1(T,\mathfrak{g}).$$

We therefore need to prove that  $$ \dim H^1(T,\mathfrak{g})>\dim H^1(T,\mathfrak{h}),$$  and then the result  will follow from Theorem \ref{t:key}.  Noting that the exponents of $\uH$ form a subset of the exponents of $\uG$ (see (\ref{e:exp})), we now argue as in Claim \ref{c:cla} in the proof of Theorem \ref{t:twosteps}, adapting it to our two cases.

Let us first treat the case where $X=D_4$. By (\ref{e:exp}) $E=\{3\}$ and so $R_{B_3,D_4}=2$, while $L_{(n_1,n_2,n_3),B_3,D_4}=\sum_{k=1}^3 \lfloor 3/n_k \rfloor$. Since we are excluding the case $n_2=3$, one now easily checks that inequality (\ref{e:crucial}) holds.

Let us now treat the case where $X=A_6$. By (\ref{e:exp}) $E=\{2,4,6\}$ and so $R_{B_3,A_6}=9$, while $L_{(n_1,n_2,n_3),B_3,A_6}=\sum_{k=1}^3(\lfloor 2/n_k \rfloor +\lfloor 4/n_k \rfloor+\lfloor 6/n_k \rfloor)$. Since we are excluding the cases $(n_1,n_2)=(2,3)$ and $(n_1,n_2,n_3)\in \{(2,4,5),(2,4,6)\}$, one now easily checks that inequality (\ref{e:crucial}) holds.
\end{proof}

Theorem \ref{t:main} follows by careful bookkeeping from  the results in this section.

\section{Tables}\label{s:tab}
For convenience, we regroup Tables \ref{tab:exrep}, \ref{tab:twosteps} and \ref{tab:twostepssmall} mentioned respectively in Theorems \ref{t:exrep}, \ref{t:twosteps} and  \ref{t:twostepssmall}.

 \begin{table}[h]
\caption{Cases for which $\dim H^1(T,\mathfrak{g})=0$ (in Theorem \ref{t:exrep})}\label{tab:exrep}
\center
\begin{tabular}{|l|l|}
\hline
$X$ & $(a,b,c)$\\
\hline
$X=A_2$& $a=2$\\
$X=C_2$ &  $a \in \{2,3\}$, $b=3$\\
$X=G_2$ &$a=2$, $b\in\{4,5\}$, $c=5$\\
\hline
\end{tabular}
\end{table}

\begin{table}
\caption{Cases for which $\dim H^1(T,\mathfrak{h})=H^1(T,\mathfrak{g)}$ (in Theorem \ref{t:twosteps})}\label{tab:twosteps}
\center
\begin{tabular}{|l|l|l|}
\hline
$X$& $(a,b,c)$ & $r$\\
\hline
$A_r$  & $(2,3,7)$ & $r \leq 19$ \\
($r \geq 3$, $r\neq 6$)& $(2,3,8)$ & $r \leq 13$\\
& $(2,3,c)$, $c \geq 9$ & $r\leq 7$\\
& $(2,4,5)$ & $r\leq 13$\\
& $(2,4,c)$, $c \geq 7$& $r \leq 5$\\
& $(2,b,c)$, $b \geq 5$ & $r=3$\\
\hline
 $D_r$ & $(2,3,7)$ & $r \in\{5,7,8,9,10,11,13,15,16,17,19,22,23,25,29,31,37,43\}$\\
$(r\geq 5)$& $(2,3,8)$& $r\in \{5,7,9,10,11,13,17,19,25\}$ \\
& $(2,3,9)$ & $r\in \{5,7,10,11,13,19\}$\\
& $(2,3,10)$ & $r\in \{5,7,11,13\}$\\
& $(2,3,11)$& $r \in\{5,7,13\}$\\
& $(2,3,12)$& $r \in\{5,7,13\}$\\
& $(2,3,c)$, $c\geq 13$ & $r\in \{5,7\}$\\
& $(2,4,5)$ & $r \in \{5,6,7,9,11,13,17,21\}$\\
& $(2,4,7)$ & $r \in \{5,9\}$\\
& $(2,4,8)$ & $r\in \{5,9\}$\\
& $(2,4,c)$, $c \geq 9$ & $r=5$\\
& $(2,5,5)$ & $r \in \{6,7,11\}$\\
& $(2,5,6)$ & $r=7$\\
& $(3,3,4)$ & $r \in\{5,7,10,13\}$\\
& $(3,3,5)$ & $r =7$\\
& $(3,3,6)$ & $r =7$\\
& $(4,4,4)$& $r=5$\\
\hline
\end{tabular}
\end{table}

\begin{table}
\caption{Cases for which $\dim H^1(T,\mathfrak{h})=H^1(T,\mathfrak{g)}$ (in Theorem \ref{t:twostepssmall})}\label{tab:twostepssmall}
\center
\begin{tabular}{|l|l|}
\hline
$X$& $(a,b,c)$ \\
\hline
$B_3$ & $(2,3,c)$, $c \geq 7$ \\
& $(3,3,c)$, $c \geq 4$     \\
\hline
$E_6$ & $(2,3,7), (2,3,8), (2,4,5), (2,4,7),(2,4,8)$\\
\hline
\end{tabular}
\end{table}

\end{document}